\documentclass[11pt]{amsart}

\usepackage{tikz,amsmath,amssymb}
\usepackage{hyperref}
\hypersetup{
    colorlinks=true,
    linkcolor=blue,
    citecolor=magenta,
    filecolor=magenta,      
    urlcolor=magenta,
}
\usepackage{float}
\usepackage{latexsym,bbm,mathptmx,microtype,cite}

\newtheorem{theorem}{Theorem}[section]
\newtheorem{prop}[theorem]{Proposition}

\newtheorem{lemma}[theorem]{Lemma}
\newtheorem{corollary}[theorem]{Corollary}

\theoremstyle{definition}

\newtheorem{definition}[theorem]{Definition}

\newtheorem{example}[theorem]{Example}

\makeatletter
\newtheorem*{rep@theorem}{\rep@title}\newcommand{\newreptheorem}[2]{%
\newenvironment{rep#1}[1]{%
\def\rep@title{\bf #2 \ref{##1}}%
\begin{rep@theorem}}%
{\end{rep@theorem}}}
\makeatother
\newreptheorem{theorem}{Theorem}

\renewcommand\emptyset{\varnothing}

\newcommand{\conv}[1]{\mathrm{conv}\left \{ #1\right\}}

\newcommand{\R}{\mathbb{R}}
\newcommand{\Z}{\mathbb{Z}}

\newcommand{\Q}{\mathbb{Q}}
\newcommand{\bv}{\mathbf{v}}

\newcommand{\bw}{\mathbf{w}}
\newcommand{\bW}{\mathbf{W}}
\newcommand{\br}{\mathbf{r}}
\newcommand{\ba}{\mathbf{a}}
\newcommand{\bp}{\mathbf{p}}

\newcommand{\vol}{\mathrm{vol}}

\newcommand{\Ehr}{\mathrm{Ehr}}

\newcommand{\cone}[1]{\mathrm{cone}\left(#1\right)}
\newcommand{\boxx}[1]{\mathrm{Box}\left(#1\right)}

\newcommand\commentout[1]{}

\begin{document}

\title{Decompositions of Ehrhart $h^*$-polynomials for rational polytopes}
\author{Matthias Beck}
\address{Department of Mathematics, San Francisco State University \& Mathematisches Institut, Freie Universit\"at Berlin\\
\url{http://math.sfsu.edu/beck/}}
\email{mattbeck@sfsu.edu}

\author{Benjamin Braun}
\address{Department of Mathematics\\
         University of Kentucky\\
\url{https://sites.google.com/view/braunmath/}}
\email{benjamin.braun@uky.edu}

\author{Andr\'es R. Vindas-Mel\'endez}
\address{Department of Mathematics\\
         University of Kentucky\\
\url{https://ms.uky.edu/\~arvi222}}
\email{andres.vindas@uky.edu}

\date{\today}

\begin{abstract}
The Ehrhart quasipolynomial of a rational polytope $P$ encodes the number of integer lattice
points in dilates of $P$, and the $h^*$-polynomial of $P$ is the numerator of the
accompanying generating function.
We provide two decomposition formulas for the $h^*$-polynomial of a rational polytope.  
The first decomposition generalizes a theorem of Betke and McMullen for lattice polytopes.  
We use our rational Betke--McMullen formula to provide a novel proof of Stanley's Monotonicity Theorem for the $h^*$-polynomial of a rational polytope. 
The second decomposition generalizes a result of Stapledon, which we use to provide rational
extensions of the Stanley and Hibi inequalities satisfied by the coefficients of the $h^*$-polynomial for lattice polytopes.  
Lastly, we apply our results to rational polytopes containing the origin whose duals are lattice polytopes. 
\end{abstract}

\maketitle


\section{Introduction}\label{sec:intro}

For a $d$-dimensional rational polytope $P \subset \R^{d}$ (i.e., the convex hull of finitely
many points in $\Q^d$) and a positive integer $t$, let $L_P(t)$ denote the number of integer lattice points in $tP$. 
Ehrhart's theorem~\cite{Ehrhart} tells us that $L_P(t)$ is of the form $\vol(P) \,
t^d+k_{d-1}(t) \, t^{d-1}+\cdots+k_1(t)\, t+k_0(t)$, where $k_0(t),k_1(t),\dots,k_{d-1}(t)$ are periodic functions in $t$.
We call $L_P(t)$ the \emph{Ehrhart quasipolynomial} of $P$, and Ehrhart proved that each
period of $k_0(t),k_1(t),\dots,k_{d-1}(t)$ divides
the \emph{denominator} $q$ of $P$, which is the  least  common  multiple  of all its  vertex  coordinate denominators.  
The \emph{Ehrhart series} is the rational generating function
\[
  \Ehr(P;z):=\sum_{t\geq 0}L(P;t) \, z^t=\frac{h^*(P;z)}{(1-z^q)^{d+1}} \, ,
\]
where $h^*(P;z)$ is a polynomial of degree less than $q(d+1)$, the \emph{$h^*$-polynomial}
of~$P$.\footnote{
Note that the $h^*$-polynomial depends not only on $q$ (though that is implicitly determined by $P$), but
also on our choice of representing the rational function $\Ehr(P;z)$, which in our form will not be in
lowest terms.
} 

Our first main contributions are generalizations of two well-known decomposition formulas of the $h^*$-polynomial for lattice polytopes due to Betke--McMullen~\cite{BetkeMcMullen} and Stapledon~\cite{StapledonInequalities}.
(All undefined terms are specified in the sections below.)

\begin{reptheorem}{decomp}
  For a triangulation $T$ with denominator $q$ of a rational $d$-polytope $P$,
  \[
    \Ehr(P;z)=\frac{\sum_{\Omega\in T}B(\Omega;z) \, h(\Omega;z^q)}{(1-z^q)^{d+1}} \, .
  \]
\end{reptheorem}

\begin{reptheorem}{h^*-polynomial}
Consider a rational $d$-polytope $P$ that contains an interior point $\frac{\ba}{\ell}$, where $\ba \in \Z^d$ and $\ell\in \Z_{>0}$. 
Fix a boundary triangulation $T$ of $P$ with denominator $q$.
Then  
$$h^*(P;z)=\frac{1-z^q}{1-z^\ell} \sum_{\Omega\in T}\left(B(\Omega;z)+B(\Omega';z)\right)
h(\Omega;z^q) \, .$$ 
\end{reptheorem}

Our second main result is a generalization of inequalities provided by Hibi~\cite{HibiEhrhartPolynomials} and Stanley~\cite{StanleyHilbertFunctionCohenMacaulay} that are satisfied by the coefficients of the $h^*$-polynomial for lattice polytopes.

\begin{reptheorem}{inequalities}
Let $P$ be a rational $d$-polytope with denominator $q$ and let $s:=\deg{h^*(P;z)}$ . 
The $h^*$-vector $(h^*_0,\dots,h^*_{q(d+1)-1})$ of $P$ satisfies the following inequalities: 
\begin{align}
      h^*_0+\cdots+h^*_{i+1}\geq h^*_{q(d+1)-1}+\cdots+h^*_{q(d+1)-1-i} \, , \qquad &i=0,\dots,
\left\lfloor \frac{q(d+1)-1}{2} \right\rfloor -1 \, , \label{ineq1} \\
    h^*_s+\cdots+h^*_{s-i}\geq h^*_0+\cdots+h^*_i \, , \qquad &i=0,\dots, q(d+1)-1 \, . \label{ineq2}
\end{align}
\end{reptheorem}

Inequality (\ref{ineq1}) is a generalization of a theorem by Hibi~\cite{HibiEhrhartPolynomials} for lattice polytopes, and (\ref{ineq2}) generalizes an inequality given by Stanley~\cite{StanleyHilbertFunctionCohenMacaulay} for lattice polytopes, namely the case when $q=1$. 
Both inequalities follow from the \emph{$a/b$-decomposition} of the $\overline{h^*}$-polynomial for
rational polytopes given in Theorem~\ref{h-star decomp} in Section \ref{sec:boundarytriangulations}, which in turn generalizes results (and
uses rational analogues of techniques) by Stapledon~\cite{StapledonInequalities}. 
Stapledon's $a/b$-decomposition has been used by different authors to study connections to unimodality, dilated polytopes, open polytopes, order polytopes, and connections to chromatic polynomials~\cite{BeckJochemkoMcCullough, JochemkoReal-Rootedness, JochemkoSymmetric, Leon}.

This paper is structured as follows.
In Section~\ref{sec:setup} we provide notation and background.
In Section~\ref{sec:betkemcmullen} we prove Theorem~\ref{decomp} and use this to give a novel
proof of Stanley's Monotonicity Theorem.
In Section~\ref{sec:stapledon} we prove Theorems~\ref{h^*-polynomial} and~\ref{inequalities}.
We conclude in Section~\ref{sec:applications} with some applications.


\section{Set-Up and Notation}\label{sec:setup}

A \emph{pointed simplicial cone} is a set of the form 
$$K(\bW)=\left\{\sum_{i=1}^n \lambda_i{\bw_i}:\lambda_i\geq 0\right\},$$
where $\bW:=\{\bw_1,\dots,\bw_n\}$ is a set of $n$ linearly independent vectors in $\R^d$.
If we can choose $\bw_i\in \Z^d$ then $K(\bW)$ is a \emph{rational} cone and we assume this throughout this paper. 
Define the \emph{open parallelepiped} associated with $K(\bW)$  as \begin{equation}
  \boxx{\bW}:=\left\{\sum_{i=1}^n\lambda_i\bw_i:0<\lambda_i<1\right\}.  
\end{equation}
Observe that we have the natural involution 
$\iota: \boxx{\bW}\cap \Z^{d} \rightarrow \boxx{\bW}\cap \Z^{d}$
given by
\begin{equation}\label{eq:2}
    \iota\left(\sum_i \lambda_i \bw_i\right) := \sum_i(1-\lambda_i)\bw_i \, .
    \end{equation}
We set $\boxx{\{0\}}:=\{0\}$.

Let $u:\R^{d}\rightarrow \R$ denote the projection onto the last coordinate.
We then define the \emph{box polynomial} as 
\begin{equation}
    B(\bW;z):=\sum_{\bv\in \boxx{\bW}\cap \Z^{d}} z^{u(\bv)}.
\end{equation}
If $\boxx{\bW}\cap \Z^{d}=\emptyset$, then we set $B(\bW;z)=0$. 
We also define $B(\emptyset;z)=1$.

\begin{example}
Let $\bW=\{(1,3),(2,3)\}$.
Then $$\boxx{\bW}=\{\lambda_1(1,3)+\lambda_2(2,3):0<\lambda_1,\lambda_2<1\}.$$
Thus $\boxx{\bW} \cap \Z^2=\{(1,2),(2,4)\}$ and its associated box polynomial is 
$$B(\bW;z)=z^2+z^4.$$
\end{example}

\begin{lemma}\label{lemma:box}
$\displaystyle B(\bW;z)=z^{\sum_i u({\bw_i})}B\left(\bW;\tfrac{1}{z}\right).$
\end{lemma}

\begin{proof}
Using the involution $\iota$,
\begin{align*}
z^{\sum_i u({\bw_i})}B\left(\bW;\frac{1}{z}\right)&=\sum_{\bv\in \boxx{\bW}\cap \Z^{d}}z^{\sum_i u({\bw_i})-u(\bv)}
=\sum_{\bv\in \boxx{\bW}\cap \Z^{d}}z^{u(\iota(\bv))}=B(\bW;z) \, . \qedhere
\end{align*}
\end{proof}

Next, we define the \emph{fundamental parallelepiped} $\Pi(\bW)$ to be a half-open variant of $\boxx{\bW}$, namely, 
$$\Pi(\bW):=\left\{\sum_{i=1}^n\lambda_i\bw_i:0\leq\lambda_i<1\right\}.  $$

We also want to cone over a polytope $P$. 
If $P\subset \R^d$ is a rational polytope with vertices $\bv_1,\dots,\bv_n\in \Q^d$, we lift the vertices into $\R^{d+1}$ by appending a $1$ as the last coordinate. 
Then 
\begin{equation}
    \cone{P}=\left\{\sum_{i=1}^n\lambda_i(\bv_i,1): \lambda_i\geq 0\right\}\subset \R^{d+1}.
\end{equation}
We say a point is at \emph{height} $k$ in the cone if the point lies on $\cone{P} \cap \{\mathbf{x}:x_{d+1}=k\}$. 
Note that $qP$ is embedded in $\cone{P}$ as $\cone{P}\cap \{\mathbf{x}:x_{d+1}=q\}$.

A \emph{triangulation} $T$ of a $d$-polytope $P$ is a subdivision of $P$ into simplices (of all dimensions).
If all the vertices of $T$ are rational points, define the \emph{denominator} of $T$ to be the least common multiple of all the vertex coordinate denominators of the faces of $T$.
For each $\Delta \in T$, we define the \emph{$h$-polynomial} of $\Delta$ with respect to $T$ as
\begin{equation}\label{h-vector}
   h_T(\Delta;z):=(1-z)^{d-\dim(\Delta)}\sum_{\Delta \subseteq \Phi \in T} \left(\frac{z}{1-z}\right)^{\dim(\Phi)-\dim(\Delta)},
\end{equation}
where the sum is over all simplices $\Phi\in T$ containing $\Delta$.
When $T$ is clear from context, we omit the subscript. 
Note that when $T$ is a boundary triangulation of $P$, the definition of the $h$-vector will be adjusted according to dimension, that is, $d$ should be replaced by $d-1$ in (\ref{h-vector}).

For a $d$-simplex $\Delta$ with denominator $p$, let $\bW$ be the set of ray generators of $\cone{\Delta}$ at height $p$, which are all integral. 
We then define the $h^*$-\emph{polynomial of $\Delta$} as the generating function of the last coordinate of integer points in $\Pi(\bW):=\Pi(\Delta)$, that is, $$h^*(\Delta;z)=\sum_{\bv\in \Pi(\Delta)\cap \Z^{d+1}} z^{u(\bv)}.$$
With this consideration, the Ehrhart series of $\Delta$ can be expressed as $$\Ehr(\Delta;z)=\frac{h^*(\Delta;z)}{(1-z^p)^{d+1}}.$$

We use a modified convention when $\Delta$ is a rational $m$-simplex of a triangulation $T$, where $T$ has denominator $q$.
In this case, it is possible that the denominator of $\Delta$ as an individual simplex might be different from $q$, but for coherence among all simplices in $T$ we use $q$ to select the height of the ray generators in $\Delta$.
Namely, we let $\bW=\{(\br_1,q),\dots, (\br_{m+1},q)\}$, where the $(\br_i,q)$ are integral ray generators of $\cone{\Delta}$ at height $q$.
The corresponding $h^*$-polynomial of $\Delta$ is a function of $q$ and the Ehrhart series of $\Delta$ can be expressed as $$\Ehr(\Delta;z)=\frac{h^*(\Delta;z)}{(1-z^q)^{m+1}}.$$
We may think of $h^*(\Delta;z)$ as computed via $\sum_{\bv\in \Pi(\bW)\cap \Z^{d+1}}z^{u(\bv)}$.


\section{Rational Betke--McMullen Decomposition}\label{sec:betkemcmullen}

\subsection{Decomposition \`a la Betke--McMullen}
Let $P$ be a rational $d$-polytope and $T$ a triangulation of $P$ with denominator $q$.
For an $m$-simplex $\Delta \in T$, let $\bW=\{(\br_1,q),\dots, (\br_{m+1},q)\}$, where the $(\br_i,q)$ are the integral ray generators of $\cone{\Delta}$ at height $q$ as above.  
Further, set $B(\bW;z)=:B(\Delta;z)$ and similarly $\boxx{\bW}=:\boxx{\Delta}$.
We emphasize that the $h^*$-polynomial, fundamental parallelepiped, and  box polynomial of $\Delta$ depend on the denominator $q$ of $T$.

A point $\bv\in \cone{\Delta}$ can be uniquely expressed as $\bv=\sum_{i=1}^{m+1} \lambda_i (\br_i,q)$ for $\lambda_i \geq 0$. 
Define 
\begin{equation}\label{I-bar}
    I(\bv):=\{i\in [m+1]: \lambda_i\in \Z\}
    \qquad \text{ and } \qquad
    \overline{I(\bv)}:=[m+1]\setminus I,
\end{equation}
where $[m+1]:=\{1,\dots,m+1\}$.

\begin{lemma}\label{lemma:h-star}
Fix a triangulation $T$ with denominator $q$ of a rational $d$-polytope $P$ and let $\Delta \in T$. Then $h^*(\Delta;z)=\sum_{\Omega \subseteq \Delta} B(\Omega;z).$
\end{lemma}

\begin{proof}

First we show that $\Pi(\Delta)=\biguplus_{\Omega\subseteq \Delta} \boxx{\Omega}$.
The reverse containment follows from the fact that any element in $\boxx{\Omega}$ is a linear combination of the ray generators of $\cone{\Omega}$. 

For the forward containment, if $\bv\in \Pi(\Delta)$, then 
$$\bv=\sum_{i=1}^{m+1}\lambda_i(\br_i,q)
=\sum_{i\in \overline{I(\bv)}}\lambda_i(\br_i,q)  \in \boxx{\Omega},$$ 
for $\Omega := \conv{\frac{\br_i}{q} :i\in \overline{I(\bv)}} \subseteq \Delta$. 
Note that $\bv$ will always lie in a unique $\boxx{\Omega}$ because every $\Omega$ corresponds to a different subset of $[m+1]$, which also tells us that the union we desire is disjoint. 

Thus $\Pi(\Delta)=\biguplus_{\Omega\subseteq \Delta} \boxx{\Omega}$, and so
\begin{equation*}
       h^*(\Delta;z)=\sum_{\bv\in \Pi(\Delta)\cap \Z^{d+1}}z^{u(\bv)}
    =\sum_{\Omega \subseteq \Delta}\sum_{\bv\in \boxx{\Omega}\cap \Z^{d+1}}z^{u(\bv)}
    =\sum_{\Omega \subseteq \Delta} B(\Omega;z) \, . \qedhere
\end{equation*}
\end{proof}

\begin{theorem}\label{decomp}
  For a triangulation $T$ with denominator $q$ of a rational $d$-polytope~$P$, 
  \[
    \Ehr(P;z)=\frac{\sum_{\Omega\in T}B(\Omega;z)\, h(\Omega;z^q)}{(1-z^q)^{d+1}} \, .
  \]
\end{theorem}

\begin{proof}
We write $P$ as the disjoint union of all open nonempty simplices in $T$ and use Ehrhart--Macdonald reciprocity~\cite{Ehrhart, Macdonald}:
\begin{align*}
    \Ehr(P;z)&=1+\sum_{\Delta\in T\setminus\{\emptyset\}}\Ehr(\Delta^\circ;z)
     =1 + \sum_{\Delta\in T\setminus\{\emptyset\}}(-1)^{\dim(\Delta)+1}\Ehr\left(\Delta;\frac{1}{z}\right)\\
    &=1 + \sum_{\Delta\in T\setminus\{\emptyset\}}(-1)^{\dim(\Delta)+1}\frac{h^*\left(\Delta;\frac{1}{z}\right)}{\left(1-\frac{1}{z^q}\right)^{\dim(\Delta)+1}}
     = 1 + \sum_{\Delta\in
T\setminus\{\emptyset\}}\frac{(z^q)^{\dim(\Delta)+1}(1-z^q)^{d-\dim(\Delta)} \, h^*\left(\Delta;\frac{1}{z}\right)}{(1-z^q)^{d+1}}.
\end{align*}
Note that the Ehrhart series of each $\Delta$ is being written as a rational function with denominator $(1-z^q)^{d+1}$. 
Using Lemma~\ref{lemma:h-star}, 
\begin{align*}\Ehr(P;z)&=1 + \sum_{\Delta\in T\setminus\emptyset}\frac{(z^q)^{\dim(\Delta)+1}(1-z^q)^{d-\dim(\Delta)}\sum_{\Omega \subseteq \Delta} B\left(\Omega;\frac{1}{z}\right)}{(1-z^q)^{d+1}}\\
&= \frac{\sum_{\Delta\in T}\left[(z^q)^{\dim(\Delta)+1}(1-z^q)^{d-\dim(\Delta)}\sum_{\Omega \subseteq \Delta} B\left(\Omega;\frac{1}{z}\right)\right]}{(1-z^q)^{d+1}}.
\end{align*} 
By Lemma~\ref{lemma:box}, 
\begin{align*}
    h^*(P;z)&=\sum_{\Delta\in T}\left[ (z^q)^{\dim(\Delta)+1}(1-z^q)^{d-\dim(\Delta)}\sum_{\Omega \subseteq \Delta} B\left(\Omega;\frac{1}{z}\right)\right]\\
    &=\sum_{\Delta\in T}\left[ (z^q)^{\dim(\Delta)+1}(1-z^q)^{d-\dim(\Delta)}\sum_{\Omega \subseteq \Delta} (z^q)^{-\dim(\Omega)-1}B(\Omega;z)\right] \\
    &= \sum_{\Omega\in T}\sum_{\Omega \subseteq \Delta}(z^q)^{\dim(\Delta)-\dim(\Omega)}(1-z^q)^{d-\dim(\Delta)}B(\Omega;z)\\
    &= \sum_{\Omega\in T}\left[ B(\Omega;z)(1-z^q)^{d-\dim(\Omega)}\sum_{\Omega \subseteq \Delta}\left(\frac{z^q}{1-z^q}\right)^{\dim(\Delta)-\dim(\Omega)}\right].
\end{align*}
Using the definition of the $h$-polynomial, the theorem follows. 
\end{proof}

\subsection{Rational $h^*$-Monotonicity}

We now show how the following theorem follows from our rational Betke--McMullen formula.

\begin{theorem}[Stanley Monotonicity~\cite{StanleyMonotonicity}]\label{thm:StanleyMonotonicity}
Suppose that $P\subseteq Q$ are rational polytopes with $qP$ and $qQ$ integral (for minimal possible $q\in
\Z_{>0}$). Define the $h^*$-polynomials via
\[
  \Ehr(P; z) = \frac{ h^*(P;z) }{ \left( 1-z^q \right)^{ \dim(P) + 1 } } 
  \qquad \text{ and } \qquad
  \Ehr(Q; z) = \frac{ h^*(Q;z) }{ \left( 1-z^q \right)^{ \dim(Q) + 1 } } \, .
\]
Then $h^*_i(P;z)\leq h^*_i(Q;z)$ coefficient-wise. 
\end{theorem}

In addition to Stanley's original proof, Beck and Sottile~\cite{BeckSottile} provide a proof of Theorem~\ref{thm:StanleyMonotonicity} using irrational decompositions of rational polyhedra.
In the case of lattice polytopes, Jochemko and Sanyal~\cite{JochemkoSanyal} prove Theorem~\ref{thm:StanleyMonotonicity}  using combinatorial positivity of translation-invariant valuations and  Stapledon~\cite{StapledonGeometric} gives a geometric interpretation of Theorem~\ref{thm:StanleyMonotonicity} by considering the $h^*$-polynomials of lattice polytopes in terms of orbifold Chow rings.
The following lemma assumes familiarity with Cohen--Macaulay complexes and related theory; see~\cite{StanleyGreenBook} for definitions and further reading.

\begin{lemma}\label{h-vector monotonicity}
Suppose $P$ is a polytope and $T$ a triangulation of $P$.
Let $P\subseteq Q$ be a polytope and $T'$ a triangulation of $Q$ such that $T'$ restricted to $P$ is $T$.
Further, if $\dim(P)<\dim(Q)$, assume that there exists a set of affinely independent vertices $\bv_1,\ldots,\bv_n$ of $Q$ outside the affine span of $P$ such that (1) the join $T\ast \conv{\bv_1,\ldots,\bv_n}$ is a subcomplex of $T'$ and (2) $\dim(P\ast \conv{\bv_1,\ldots,\bv_n})=\dim(Q)$.
For every face $\Omega\in T$, the coefficient-wise inequality $h_T(\Omega;z)\leq
h_{T'}(\Omega, z)$ holds. 
\end{lemma}

\begin{proof}
Suppose first that $\dim(P)=\dim(Q)$. 
Let $T$ be a triangulation of $P$ and $T'$ a triangulation of $Q$ such that $T'$ restricted to $P$ is $T$. 
Note that $T$ and $T'$ are geometric simplicial complexes covering $P$ and $Q$, respectively. 
Let $\Omega\in T$. 
Then $\operatorname{link}_T(\Omega)$ and $\operatorname{link}_{T'}(\Omega)$ are either balls or spheres, hence Cohen--Macaulay.
Now, consider $\mathcal{R}:=\operatorname{link}_{T'}(\Omega)-\operatorname{link}_T(\Omega)$, which is a relative simplicial complex.  
By~\cite[Corollary 7.3(iv)]{StanleyGreenBook} $\mathcal{R}$ is also Cohen--Macaulay.  
From~\cite[Proposition 7.1]{StanleyGreenBook} it follows that
\[
  h_{\mathcal{R}}(\emptyset;z)=h_{T'}(\Omega;z)-h_T(\Omega;z)
  \qquad \text{ and } \qquad
  h_{\mathcal{R}}(\emptyset;z), \ h_T(\Omega;z), \ h_{T'}(\Omega;z)\geq 0 \, . 
\]
Rearranging, we obtain that $h_{T'}(\Omega;z)=h_{\mathcal{R}}(\emptyset;z)+h_T(\Omega;z)$, which implies that $h_{T}(\Omega;z)\leq h_{T'}(\Omega;z)$
Hence, for each face in $T$, the result follows. 

Now, consider the case when $\dim(P)<\dim(Q)$.
Again, let $T$ be a triangulation of $P$ and $T'$ a triangulation of $Q$ such that $T'$ restricted to $P$ is $T$, where we further assume that there exists a set of affinely independent vertices $\bv_1,\ldots,\bv_n$ of $Q$ outside the affine span of $P$ such that (1) the join $T\ast \conv{\bv_1,\ldots,\bv_n}$ is a subcomplex of $T'$ and (2) $\dim(P\ast \conv{\bv_1,\ldots,\bv_n})=\dim(Q)$.
\color{black}
Note that the affine independence of the $\bv_i$'s implies that
\[
  \dim(\conv{P\cup\bv_1\cup\cdots \cup \bv_k})=\dim(\conv{P\cup\bv_1\cup\cdots \cup \bv_{k-1}})+1 \, .
\]
Let $T_k$ denote the join of $T$ with the simplex $\conv{\bv_1,\ldots,\bv_k}$.
Let $\Omega\in T_k$.
Since $\Omega \subseteq \partial T_{k+1}$ and $\operatorname{link}_{T_k}(\Omega)$ and $\operatorname{link}_{T_{k+1}}(\Omega)$ are both balls, $\mathcal{R}:=\operatorname{link}_{T_{k+1}}(\Omega)-\operatorname{link}_{T_k}(\Omega)$ is Cohen--Macaulay by~\cite[ Proposition 7.3(iii)]{StanleyGreenBook}.
Thus, by a similar argument as given in the paragraph above,
\[
h_{T_k}(\Omega;z)\leq h_{T_{k+1}}(\Omega;z) \, .
  \]
  Combining this with the fact that $\dim(P\ast \conv{\bv_1,\ldots,\bv_n})=\dim(Q)$, it follows by induction (for the first inequality) and our previous case (for the second inequality) that for $\Omega\in T$
  \[
h_T(\Omega;z)\leq h_{T_n}(\Omega;z)\leq h_{T'}(\Omega;z)\, . \qedhere
    \]
\end{proof}

\begin{proof}[Proof of Theorem~\ref{thm:StanleyMonotonicity}]
Let $P$ be a polytope contained in $Q$. 
Let $T$ be a triangulation of $P$ and let $T'$ be a triangulation of $Q$ such that $T'$ restricted to $P$ is $T$, where if $\dim(P)<\dim(Q)$ the triangulation $T'$ satisfies the conditions given in Lemma~\ref{h-vector monotonicity}.
(Note that such a triangulation $T'$  can always be obtained from $T$, e.g., by extending $T$ using a placing triangulation.)
By Theorem~\ref{decomp}, $h^*(P;z)=\sum_{\Omega\in T}B(\Omega;z) \, h_T(\Omega;z^q)$. 
Since $P$ is contained in $Q$,
$$h^*(Q;z)=\sum_{\Omega\in T}B(\Omega;z) \, h_{T'|_P}(\Omega;z^q)+\sum_{\Omega\in T'\setminus
T}B(\Omega;z) \, h_{T'}(\Omega;z^q).$$
By Lemma~\ref{h-vector monotonicity}, the coefficients of $\sum_{\Omega\in T}B(\Omega;z)h_{T'}(\Omega;z^q)$ dominate the coefficients of $\sum_{\Omega\in T}B(\Omega;z)h_T(\Omega;z^q).$
This further implies that the coefficients of $h^*(Q;z)$ dominate the coefficients of $h^*(P;z)$ since 
\begin{align*}\sum_{\Omega\in T}B(\Omega;z) \, h_T(\Omega;z^q)&\leq \sum_{\Omega\in
T}B(\Omega;z) \, h_{T'}(\Omega;z^q) \\
&\leq \sum_{\Omega\in T}B(\Omega;z) \, h_{T'|_P}(\Omega;z^q)+\sum_{\Omega\in T'\setminus T}B(\Omega;z)
\, h_{T'}(\Omega;z^q) \, . \qedhere
\end{align*}
\end{proof}


\section{$h^*$-Decompositions from Boundary Triangulations}\label{sec:stapledon}\label{sec:boundarytriangulations}

\subsection{Set-up}\label{setup}

Throughout this section we will use the following set-up.
Fix a boundary triangulation $T$ with denominator $q$ of a rational $d$-polytope $P$.
 Take $\ell \in \Z_{>0}$, such that $\ell P$ contains a lattice point $\ba$ in its interior.  
Thus $(\ba,\ell)\in \cone{P}^\circ \cap \Z^{d+1}$ is a lattice point in the interior of the cone of $P$ at height $\ell$, and $\cone{(\ba,\ell)}$ is the ray through the point $(\ba,\ell)$.
We cone over each $\Delta \in T$ and define $\bW=\{(\br_1,q),\dots, (\br_{m+1},q)\}$ where the $(\br_i,q)$ are integral ray generators of $\cone{\Delta}$ at height $q$.
As before, we have the associated box polynomial $B(\bW;z)=:B(\Delta;z)$.
Now, let $\bW'=\bW \cup \{(\ba,\ell)\}$ be the set of generators from $\bW$ together with $(\ba,\ell)$ and we set $\cone{\Delta'}$ to be the cone generated by $\bW'$, with associated box polynomial $B(\bW';z)=:B(\Delta';z)$. 

\begin{corollary}\label{symmetry box polynomials}
For each face $\Delta$ of $T$,
\[
  B(\Delta;z)=z^{q(\dim(\Delta) +1)} B\left(\Delta;\tfrac{1}{z}\right)
  \qquad \text{ and } \qquad
  B(\Delta';z)=z^{q(\dim(\Delta)+1)+\ell} B\left(\Delta';\tfrac{1}{z}\right).
\]
\end{corollary}

\begin{proof}
The height of $\sum_i(\br_i,q)$ is $q$ times the number of summands, which gives us $q(\dim(\Delta)+1)$.
 The first equations now follow from the involution $\iota$ and Lemma~\ref{lemma:box}; note
that we will have to use $\bW$ in the first case and $\bW'$ in the second.
\end{proof}

Observe that when $\Delta=\emptyset$ is the empty face, $B(\emptyset;z)=1$, but $B(\emptyset';z)=B((\ba,\ell);z)$.
This differs from the scenario in~\cite{StapledonInequalities} where Stapledon's set-up determined that $B(\emptyset',z)=0$. 

For a real number $x$, define $\lfloor x \rfloor$ to be the greatest integer less than or equal to $x$. 
Additionally, define the \emph{fractional part} of $x$ to be $\{x\} = x - \lfloor x \rfloor$.

\subsection{Boundary Triangulations}

For each $\bv\in \cone{P}$ we associate two faces $\Delta(\bv)$ and $\Omega(\bv)$ of $T$, as follows.
The face $\Delta(\bv)$ is chosen to be the minimal face of $T$ such that $\bv \in
\cone{\Delta'(\bv)}$, and we define
$$\Omega(\bv):=\conv{\frac{\br_i}{q} :i\in \overline{I(\bv)}}\subseteq \Delta(\bv),$$
where $\overline{I(\bv)}$ is defined as in (\ref{I-bar}) and the $(\br_i,q)$ are ray generators of $\cone{\Delta(\bv)}$.
In an effort to make our statements and proofs less notation heavy, for the rest of this section we write  $\Delta(\bv)=\Delta$ and $\Omega(\bv)=\Omega$ with the understanding that both depend on $\bv$. 
Furthermore, for $\bv=\sum_{i=1}^{m+1}\lambda_i(\br_i,q)+\lambda(\ba,\ell)$ where $\lambda,
\lambda_i\geq 0$, define $$\{\bv\} := \sum_{i\in \overline{I(\bv)}} \{\lambda_i\} (\br_i,q) + \{\lambda\}(\ba,\ell).$$

\begin{lemma}\label{lemma:unique sum}
Given $\bv\in \cone{P}$, construct $\Delta = \Delta(\bv)$ as described above, with
$\cone{\Delta}$ generated by $(\br_1,q), \dots, (\br_{m+1},q)$.
Then $\bv$ can be written uniquely as 
\begin{equation}\label{v in parts}
    \{\bv\} +\sum_{i\in I(\bv)} (\br_i,q)
+\sum_{i=1}^{m+1}\mu_i(\br_i,q) +\mu(\ba,\ell),
\end{equation}
 where $\mu,\mu_i \in \Z_{\geq 0}$.
\end{lemma}

Below we will note the dependence of the unique coefficients $\mu_i$ and $\mu$ on $\bv$ by writing them as
$\mu_i(\bv)$ and~$\mu(\bv)$.

\begin{proof}
Since $\bv$ is in $\cone{\Delta'}$, it can be written as a linear combination of the generators of $\cone{\Delta}$ and $(\ba,\ell)$. 
We further express $\bv$ as a sum of its integer and fractional parts. \begin{align*}
    \bv &= \sum_{i=1}^{m+1}\lambda_i(\br_i,q) +\lambda (\ba,\ell), \text{ for } \lambda_i>0 \text{ and } \lambda \geq 0\\
    &= \sum_{i\in \overline{I(\bv)}}\{\lambda_i \} (\br_i,q)+ \{\lambda \}(\ba,\ell) + \sum_{i=1}^{m+1}\lfloor \lambda_i \rfloor (\br_i,q) + \lfloor \lambda \rfloor (\ba,\ell)\\
    &= \{\bv\} + \sum_{i=1}^{m+1}\lfloor \lambda_i \rfloor (\br_i,q) + \lfloor \lambda \rfloor (\ba,\ell). 
\end{align*}

 
Note that each $\lambda_i>0$ because of the minimality of $\Delta$.
Recall that $\Omega = \conv{\frac{\br_i}{q} :i\in \overline{I(\bv)}} \subseteq \Delta$. Thus
\begin{itemize}
    \item if $\lambda \notin \Z$, then $\{\bv\}\in \boxx{\Omega'}$,  
    \item if $\lambda \in \Z$, then $\{\bv\}\in \boxx{\Omega}$.
\end{itemize}
Further observe that when $\lambda$ is an integer, $\{\bv\}$ is an element on the boundary of $\cone{P}$. 

If $i\in I(\bv)$, then $\lambda_i\in \Z$ and $\lfloor \lambda_i \rfloor = \lambda_i \geq 1$ for $i\in I(\bv)$. 
This allows us to represent $\bv$ in the form 
$$\bv = \{\bv\} +\sum_{i\in I(\bv)} (\br_i,q)
+\sum_{i=1}^{m+1}\mu_i(\br_i,q) +\mu(\ba,\ell),$$
where $\mu,\mu_i \in \Z_{\geq 0}$.
\end{proof}

\begin{corollary}\label{cor: heights}
Continuing the notation above,
\begin{equation}
    u(\bv)=u(\{\bv\})+q(\dim{\Delta(\bv)}-\dim{\Omega(\bv)})+\sum_{i=1}^{m+1}q \,
\mu_i(\bv)+\mu(\bv) \, \ell \, .
\end{equation}
\end{corollary}

\begin{proof}
This follows from considering the height contribution of each part in (\ref{v in parts}). 
\end{proof}

The following theorem provides a decomposition of the $h^*$-polynomial of a rational
polytope in terms of box and $h$-polynomials. 
It is important to note again that the $h^*$-polynomial depends on the denominator of the boundary triangulation.

\begin{theorem}\label{h^*-polynomial}
Consider a rational $d$-polytope $P$ that contains an interior point $\frac{\ba}{\ell}$, where $\ba \in \Z^d$ and $\ell\in \Z_{>0}$. 
Fix a boundary triangulation $T$ of $P$ with denominator $q$.
Then  
$$h^*(P;z)=\frac{1-z^q}{1-z^\ell} \sum_{\Omega\in T}\left(B(\Omega;z)+B(\Omega';z)\right) h(\Omega;z^q).$$ 
\end{theorem}

\begin{proof}
By 
Corollary~\ref{cor: heights},
\begin{align*}
\frac{h^*(P;z)}{(1-z^q)^{d+1}}
    &=\sum_{\bv\in \cone{P} \cap \Z^{d+1}}z^{u(\bv)}\\
    &=\sum_{\bv \in \cone{P} \cap \Z^{d+1}} z^{u(\{\bv\})+q(\dim{\Delta(\bv)}-\dim{\Omega(\bv)})+\sum_{i=1}^{\dim(\Delta)+1}q\mu_i(\bv)+\mu(\bv) \ell}\\
    &=\sum_{\Delta\in T}\sum_{\Omega\subseteq \Delta} z^{ q(\dim{\Delta}-\dim{\Omega}) } \sum_{\bv\in
\left(\boxx{\Omega}\cup\boxx{\Omega'}\right)\cap \Z^{d+1}}z^{u(\bv)} \sum_{\mu_i,\mu \geq 0} z^{\sum_{i=1}^{\dim(\Delta)+1}q\mu_i+\mu \ell}\\
    &=\sum_{\Delta\in T}\sum_{\Omega\subseteq \Delta}\frac{\left(B(\Omega;z)+B(\Omega';z)\right)z^{q(\dim{\Delta}-\dim{\Omega})}}{(1-z^q)^{\dim(\Delta)+1}(1-z^{\ell})}\\
        &=\frac{1}{1-z^\ell}\sum_{\Omega \in T}\left(B(\Omega;z)+B(\Omega';z)\right)\sum_{\Omega\subseteq \Delta}\frac{(z^q)^{\dim(\Delta)-\dim(\Omega)}}{(1-z^q)^{\dim(\Delta)+1}}\\
          &=\frac{1}{(1-z^\ell)(1-z^q)^d}\sum_{\Omega \in
T}\left(B(\Omega;z)+B(\Omega';z)\right) h(\Omega;z^q) \, . \qedhere
\end{align*}
\end{proof}

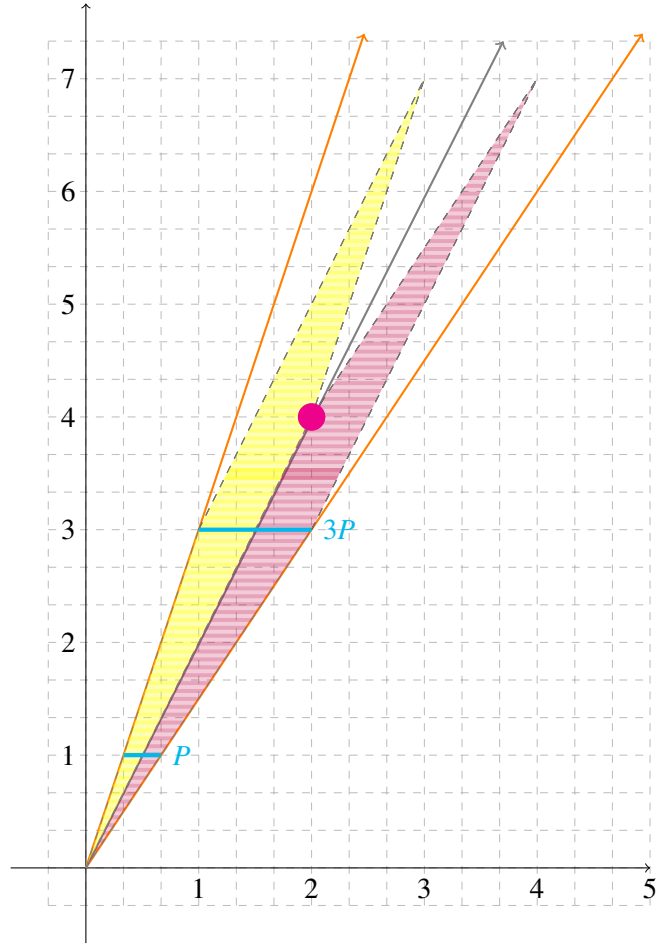
\begin{figure}
\begin{center}
\begin{tikzpicture}
 \definecolor{polytopecolor}{rgb}{0.0, 0.72, 0.92}
\draw[color=lightgray,step=.5cm,
dashed] (-0.5,-.5) grid (7.5,11);
\draw[->] (-1,0) -- (7.5,0)
node[below right] {};
\draw[->] (0,-1) -- (0,11.5)
node[left] {};

\draw[thick, orange,->] (0,0) -- (3.7, 11.1) 
node[left] {};

\draw[thick, orange,->] (0,0)  -- (7.4,11.1)
node[left] {};

\draw[thick, gray,->] (0,0) -- (5.55, 11)
node[left] {};

\draw (3/2,0)
node[below] {1};
\draw (3,0)
node[below] {2};
\draw (4.5,0)
node[below] {3};
\draw (6,0)
node[below] {4};
\draw (7.5,0)
node[below] {5};
\draw (0,3/2)
node[left] {1};
\draw (0,3)
node[left] {2};
\draw (0,4.5)
node[left] {3};
\draw (0,6)
node[left] {4};
\draw (0,7.5)
node[left] {5};
\draw (0,9)
node[left] {6};
\draw (0,10.5)
node[left] {7};

\draw (3,6)
node[fill, color=magenta,circle, minimum size] {};

\draw[dashed] (3/2,9/2) -- (4.5,10.5)--(3,6)
node[left] {};

\draw[dashed] (3,9/2) -- (6,10.5)--(3,6)
node[left] {};

\draw[dashed][gray,top color=yellow, bottom color=yellow, fill opacity=0.3] (3/2,9/2) -- (4.5,10.5)--(3,6) -- (0,0) -- cycle;

\draw[dashed][gray,top color=purple, bottom color=purple, fill opacity=0.2] (3,9/2) -- (6,10.5)--(3,6) -- (0,0) -- cycle;
node[left] {};


\draw[ultra thick, polytopecolor] (3/2,9/2) -- (3,9/2)
node[right] {$3P$};
\draw[ultra thick, polytopecolor] (1/2,3/2) -- (1,3/2) 
node[right] {$P$};

\draw (3,6)
node[fill, color=magenta,circle] {};

\end{tikzpicture}
\end{center}
\caption{ This figure shows $\cone{P}$ (in orange), $P$, $3P$, $(\ba,\ell)=(2,4)$, $\boxx{\Delta_1'}$ (in yellow), $\boxx{\Delta_2'}$ (in pink).}
\label{fig:cone} 
\end{figure}

\begin{example}
  Following the setup in Section~\ref{setup}, consider the line segment
$P=\left[\frac{1}{3},\frac{2}{3}\right]$ and so our boundary triangulation $T$ has denominator $3$.
  In the cone over $P$, set $(\ba,\ell)=(2,4)$. 
  The simplices in $T$ are the empty face $\emptyset$ and the two vertices $\Delta_1=\frac{1}{3}$ and $\Delta_2=\frac{2}{3}$.
  The cones over the vertices have integral ray generators $\bW_1=\{\left(1,3\right)\}$ and $\bW_2=\{\left(2,3\right)\}$. 
We see that if $\bv\in \cone{P}$ then the only  options for $\Delta(\bv)$ to be chosen as a minimal face of $T$ such that $\bv\in \cone{\Delta'(\bv)}$ are again to consider $\emptyset$, $\Delta_1$, and $\Delta_2$.
In this example, $\Omega(\bv)=\Delta(\bv)$. 
Recall that since $T$ is a boundary triangulation of $P$, the definition of the $h$-vector (\ref{h-vector}) is adjusted according to dimension, that is, $d$ is be replaced by $d-1$.

From Figure~\ref{fig:cone} we determine the following: 

\begin{center}
 \begin{tabular}{||c c c c c||} 
 \hline
 $\Omega\in T$ & $\dim(\Omega)$ & $B(\Omega;z)$ & $B(\Omega';z)$ & $h(\Omega,z^3)$\\ [0.5ex] 
 \hline\hline
 $\Delta_1$ & 0 & 0 & 0 & 1 \\ 
 \hline
 $\Delta_2$ & 0 & 0 & 0 & 1 \\
 \hline
 $\emptyset$ & -1 & 1 & $z^2$ & $1+z^3$ \\
 [1ex] 
 \hline
  \end{tabular}
\end{center}

Applying Theorem~\ref{h^*-polynomial}, we obtain 
\begin{align*}
    h^*(P;z)&=\frac{1-z^3}{1-z^4}\left(1+z^3+z^2+z^5\right)\\
    &=1+z^2+z^4,
\end{align*}
which agrees with the computation obtained using Normaliz~\cite{Normaliz}. 
\end{example}

\subsection{Rational Stapledon Decomposition and Inequalities}
Using Theorem~\ref{h^*-polynomial}, we can rewrite the $h^*$-polynomial of a rational polytope $P$ 
as
$$h^*(P;z)= \frac{1+z+\cdots+z^{q-1}}{1+z+\cdots+z^{\ell-1}} \sum_{\Omega\in
T}\left(B(\Omega;z)+B(\Omega';z)\right) h(\Omega;z^q)\, .
$$ 
Next, we turn our attention to the polynomial 
\begin{equation}\label{h-bar}
    \overline{h^*}(P;z):=\left(1+z+\cdots+z^{\ell-1}\right)h^*(P;z) \, .
\end{equation}
We know that $h^*(P;z)$ is a polynomial of degree at most $q(d+1)-1$, thus
$\overline{h^*}(P;z)$ has degree at most $q(d+1)+\ell-2$. 
We set $f$ to be the degree of $\overline{h^*}(P;z)$ and $s$ to be the degree of $h^*(P;z)$.
We can recover $h^*(P;z)$ from $\overline{h^*}(P;z)$ for a chosen value of $\ell$; 
if we write 
\begin{equation*}
    \overline{h^*}(P;z)=\overline{h^*_0}+\overline{h^*_1}z+\cdots + \overline{h^*_f}z^f,
\end{equation*}
then 
\begin{equation}\label{sum of h_i}
    \overline{h^*_i}=h^*_i+h^*_{i-1}+ \cdots +h^*_{i-l+1} \, \qquad i=0,\dots, f,
\end{equation}
and we set
$h^*_i=0$ when $i>s$ or $i<0$.

\begin{prop}\label{smallest dilate}
Let $P$ be a rational $d$-polytope with denominator $q$ and Ehrhart series 
$$
\Ehr(P;z)=\frac{h^*(P;z)}{(1-z^q)^{d+1}}\, .
$$
Then $\deg{h^*(P;z)}=s$ if and only if $(q(d+1)-s)P$ is the smallest integer dilate of $P$ that contains an interior lattice point.
\end{prop}

\begin{proof}
Let $L(P;t)$ and $L(P^\circ;t)$ be the Ehrhart quasipolynomials of $P$ and the interior of $P$, respectively. 
Using Ehrhart--Macdonald reciprocity~\cite{Ehrhart, Macdonald} we obtain
\begin{align*}
   \Ehr(P^\circ;z)
    &= \sum_{t\geq 1}L(P^\circ; t)z^t
     = (-1)^{d+1}\frac{\sum^s_{j=0}h^*_j\left(\frac{1}{z}\right)^j}{\left(1-\frac{1}{z^q}\right)^{d+1}}
     =z^{q(d+1)}\frac{\sum^s_{j=0}h^*_jz^{-j}}{(1-z^q)^{d+1}}\\
    &=\left(\sum^s_{j=0}h^*_jz^{q(d+1)-j}\right)(1+z^q+z^{2q}+\cdots)^{d+1}.
\end{align*}
Now, note that the minimum degree term of $$\left(\sum^s_{j=0}h^*_jz^{q(d+1)-j}\right)(1+z^q+z^{2q}+\cdots)^{d+1}$$ is $h^*_sz^{q(d+1)-s}$, which implies that the term of $\sum_{t\geq 1}L(P^\circ; t)z^t$ with minimum degree is $q(d+1)-s$. 
Hence, the degree of $h^*(P;z)$ is $s$ precisely if $(q(d+1)-s)P$ is the smallest integer dilate of $P$ that contains an interior lattice point. 
\end{proof}

The following result provides a decomposition of the $\overline{h^*}$-polynomial which we refer to as an
\emph{$a/b$-decomposition}. It generalizes \cite[Theorem~2.14]{StapledonInequalities} to the rational
case.

\begin{theorem}\label{h-star decomp}
Let $P$ be a rational $d$-polytope with denominator $q$, and let $s:=\deg{h^*(P;z)}$. 
Then $\overline{h^*}(P;z)$ has a unique decomposition
$$\overline{h^*}(P;z)=a(z)+z^\ell b(z) \, ,$$
where $\ell=q(d+1)-s$ and $a(z)$ and $b(z)$ are polynomials with integer coefficients satisfying $a(z)=z^{q(d+1)-1}a\left(\frac{1}{z}\right)$ and $b(z)=z^{q(d+1)-1-\ell}b\left(\frac{1}{z}\right)$. 
Moreover, the coefficients of $a(z)$ and $b(z)$ are nonnegative.
\end{theorem}

\begin{proof}
Let $a_i$ and $b_i$ denote the coefficients of $z^i$ in $a(z)$ and $b(z)$, respectively.  
Set 
\begin{equation}\label{a}
    a_{i+1}= h^*_0+\cdots+h^*_{i+1}-h^*_{q(d+1)-1}-\cdots-h^*_{q(d+1)-1-i},
\end{equation}
and 
\begin{equation}\label{b}
    b_i=-h^*_0-\cdots-h^*_i+h^*_s+\cdots+h^*_{s-i}.
\end{equation}
Using (\ref{sum of h_i}) and the fact that $\ell=q(d+1)-s$, we compute that 

\begin{align*}
 a_i+b_{i-\ell} &= h^*_0+\cdots+h^*_{i}-h^*_{q(d+1)-1}-\cdots-h^*_{q(d+1)-i} -h^*_0-\cdots-h^*_{i-\ell}+h^*_s+\cdots+h^*_{s-i+\ell}\\
   &= h^*_{i-\ell+1}+\cdots+h^*_i= \overline{h^*_i} \, ,\\
  a_i-a_{q(d+1)-1-i} &= h^*_0+\cdots+h^*_{i}-h^*_{q(d+1)-1}-\cdots-h^*_{q(d+1)-i}
-h^*_0-\cdots-h^*_{q(d+1)-1-i}\\
& \qquad {}+h^*_{q(d+1)-1}+\cdots+h^*_{i+1} \\
   &=0 \, ,\\
b_i-b_{q(d+1)-1-\ell-i} &= -h^*_0-\cdots-h^*_i+h^*_s+\cdots+h^*_{s-i}
 +h^*_0+\cdots+h^*_i-h^*_s-\cdots-h^*_{s-i-1}-h^*_s-\cdots-h^*_{i+1}\\
&=0 \, , 
\end{align*} for $i=0,\dots,q(d+1)-1$. 
Thus, we obtain the decomposition desired. 
The uniqueness property follows from (\ref{a}) and~(\ref{b}).

Let $T$ be a regular boundary triangulation of $P$. 
By Theorem~\ref{h^*-polynomial} and (\ref{h-bar}), we can set 
\begin{equation}
    a(z)=(1+z+\cdots+z^{q-1}) \sum_{\Omega\in T} B(\Omega;z)  \, h(\Omega;z^q) \, ,
\end{equation}
and
\begin{equation}
    b(z)=z^{-\ell}(1+z+\cdots+z^{q-1}) \sum_{\Omega\in T}B(\Omega';z)  \, h(\Omega;z^q) \, ,
\end{equation}
so that $\overline{h^*}(P;z)=a(z)+z^{\ell}b(z)$. 
By Proposition~\ref{smallest dilate}, the dilate $kP$ contains no interior lattice points for $k=1,\dots, \ell-1$, so if $\bv \in \boxx{\Omega'}\cap \Z^{d+1}$ for $\Omega \in T$, then $u(\bv)\geq \ell$. 
Hence, $b(z)$ is a polynomial. 
We now need to verify that $$a(z)=z^{q(d+1)-1}a\left(\tfrac{1}{z}\right) \qquad \text{and} \qquad
b(z)=z^{q(d+1)-1-\ell}b\left(\tfrac{1}{z}\right) . $$ 
It is a well-known property of the $h$-vector in ($\ref{h-vector}$) that $h(\Omega, z^q)=z^{q(d-1-\dim(\Omega))}h(\Omega;z^{-q})$~\cite{Fulton, McMullenShephard, StanleyFaces}. 

Using the aforementioned and Corollary~\ref{symmetry box polynomials}, we determine that 

\begin{align*}
z^{q(d+1)-1}a\left(\frac{1}{z}\right)
&= z^{q(d+1)-1} \left(1+\frac{1}{z}+\cdots+\frac{1}{z^{q-1}}\right) \sum_{\Omega\in T} B\left(\Omega;\frac{1}{z}\right) h\left(\Omega;\frac{1}{z^q}\right)\\
 &=z^{q(d+1)-1}z^{1-q}(1+z+\cdots+z^{q-1}) \sum_{\Omega\in T} B\left(\Omega;\frac{1}{z}\right) h\left(\Omega;\frac{1}{z^q}\right)\\
 &=z^{qd}(1+z+\cdots+z^{q-1}) \sum_{\Omega\in T}z^{-q(\dim(\Omega) +1)}B(\Omega,z) \, z^{-q(d-1-\dim{\Omega})}h(\Omega;z^q)\\
 &=(1+z+\cdots+z^{q-1}) \sum_{\Omega\in T}B(\Omega,z) \, h(\Omega;z^q)=a(z)
\end{align*}
and 
\begin{align*}
z^{q(d+1)-1-\ell}b\left(\frac{1}{z}\right)
   & =z^{q(d+1)-1-\ell} z^\ell \left(1+\frac{1}{z}+\cdots+\frac{1}{z^{q-1}}\right) \sum_{\Omega\in T}B\left(\Omega';\frac{1}{z}\right) h\left(\Omega;\frac{1}{z^q}\right) \\
    &= z^{q(d+1)-1}z^{1-q}(1+z+\cdots+z^{q-1})\sum_{\Omega\in T}B\left(\Omega';\frac{1}{z}\right) h\left(\Omega;\frac{1}{z^q}\right)\\
    &=z^{qd}(1+z+\cdots+z^{q-1})\sum_{\Omega\in T}  z^{-q(\dim{\Omega}+1)-\ell}B(\Omega';z) \, z^{-q(d-1-\dim{\Omega})}h(\Omega;z^q)\\
    &=z^{-\ell}(1+z+\cdots+z^{q-1})\sum_{\Omega\in T}B(\Omega';z) \, h(\Omega;z^q)=b(z) \, .
\end{align*}
Lastly, recall that the box polynomials and the $h$-polynomials have nonnegative coefficients~\cite{StanleySubdivisions}, so a sum of products of box polynomials and $h$-polynomials will also have nonnegative coefficients.
Thus, the result holds.
\end{proof}

The next theorem follows as a corollary to Theorem~\ref{h-star decomp} and gives inequalities satisfied by the coefficients of the $h^*$-polynomial for rational polytopes.

\begin{theorem}\label{inequalities}
Let $P$ be a rational $d$-polytope with denominator $q$ and let $s:=\deg{h^*(P;z)}$.
The $h^*$-vector $(h^*_0,\dots,h^*_{q(d+1)-1})$ of $P$ satisfies the following inequalities: 
\begin{align}
      h^*_0+\cdots+h^*_{i+1}\geq h^*_{q(d+1)-1}+\cdots+h^*_{q(d+1)-1-i} \, , \qquad &i=0,\dots,
\left\lfloor \frac{q(d+1)-1}{2} \right\rfloor -1 \, , \label{ineq1a} \\
    h^*_s+\cdots+h^*_{s-i}\geq h^*_0+\cdots+h^*_i \, , \qquad &i=0,\dots, q(d+1)-1 \, . \label{ineq2a}
\end{align}
\end{theorem}

\begin{proof}
By (\ref{a}) and (\ref{b}) if follows that (\ref{ineq1a}) and (\ref{ineq2a}) hold if and only
if $a(z)$ and $b(z)$ have nonnegative coefficients, respectively, which in turn follows from Theorem~\ref{h-star decomp}. 
\end{proof}


\section{Applications}\label{sec:applications}

\subsection{Rational Reflexive Polytopes}
A lattice polytope is \emph{reflexive} if its dual is also a lattice polytope.
Reflexive polytopes have enjoyed a wealth of recent research activity (see,
e.g.,~\cite{Batyrev, BraunReflexive, BraunDavisSolus, HeSeongYau,
HegedusHigashitaniKasprzyk, HibiTsuchiyaReflexiveGraphs, HibiTsuchiyaDepth,
HibiTsuchiyaReflexive, NagaokaTsuchiya}), and 
Hibi~\cite{HibiDual} proved that a lattice polytope $P$ is the translate of a reflexive polytope if and only if $\Ehr\left(P;\frac{1}{z}\right)=(-1)^{d+1}z \, \Ehr(P;z)$ as rational
functions, that is, $h^*(z)$ is palindromic.
More generally,  Fiset and Kaspryzk~\cite[Corollary 2.2]{FisetKasprzyk} proved that a rational polytope $P$ whose dual is a lattice polytope has a palindromic $h^*$-polynomial, complementing previous results by De Negri and Hibi \cite{DeNegriHibi}.
The following proposition provides an alternate route to Fiset and Kaspryzk's result. 

\begin{theorem}\label{prop: dual}
  Let $P$ be a rational polytope containing the origin.
  The dual of $P$ is a lattice polytope if and only if $\overline{h^*}(P;z)=h^*(z)=a(z)$, that is, $b(z)=0$ in the $a/b$-decomposition of $\overline{h^*}(P;z)$ from Theorem~\ref{h^*-polynomial}.
\end{theorem}

\begin{proof}
Let $P$ be a rational polytope containing the origin in its interior.  
Following Set-up~\ref{setup}, we let $T$ be a boundary triangulation of $P$ and we set $(\ba, \ell)=(\mathbf{0},1)$.
Recall that this implies
\[
  b(z)=z^{-1}(1+z+\cdots + z^{q-1})\sum_{\Omega\in T}B(\Omega';z)h(\Omega;z^q)\, .
\]
Thus, $b(z)=0$ if and only if $B(\Omega';z)=0$ for every $\Omega\in T$, which is true if and only if $\boxx{\Omega'}$ contains no integer points for every $\Omega\in T$.
  
To establish the forward direction, assume that the dual of $P$ is a lattice polytope.
We want to show that $b(z)=0$ in the $a/b$-decomposition of $\overline{h^*}(P;z) = h^*(P;z)$. 
Each $\Omega\in T$ is contained in a facet $F$ of $P$. 
Since the dual of $P$ is a lattice polytope, the vector normal to $\cone{F}$ is of the form $(\bp,1)$, where $\bp$ is the vertex of the dual of $P$ corresponding to $F$.
Let $(\br_1,q),\dots ,(\br_{m+1},q)$ be the ray generators of $\boxx{\Omega}$.
If $\sum_{i=1}^{m+1}\lambda_i(r_i,q)\in \boxx{\Omega}$ for $0<\lambda_i<1$, then $(\bp,1)\cdot \left(\sum_{i=1}^{m+1}\lambda_i(r_i,q)\right)=0$.
Also, note that $(\bp,1)\cdot (\mathbf{0},1)=1$, which tells us that $(\mathbf{0},1)$ is at lattice distance $1$ away from $\boxx{\Omega}$ with respect to $(\bp,1)$. 
So, if $$\sum_{i=1}^{m+1}\lambda_i(r_i,q)+\lambda(\mathbf{0},1)\in \boxx{\Omega'} \, $$ then $(\bp,1)\cdot \left[\sum_{i=1}^{m+1}\lambda_i(r_i,q)+\lambda(\mathbf{0},1)\right]=\lambda$, where $0<\lambda<1$. 
This implies that $\sum_{i=1}^{m+1}\lambda_i(r_i,q)+\lambda(\mathbf{0},1)$ is not an integer point, from which it follows that $\boxx{\Omega'}$ contains no lattice points. 
Thus $B(\Omega',z)=0$ and so $b(z)=0$ in the $a/b$-decomposition of $\overline{h^*}(P;z)$. 
Hence, $\overline{h^*}(P;z)= h^*(P;z)=a(z)$ is palindromic.

For the backward direction, assume that $b(z)=0$, and thus for every $\Omega\in T$, the set $\boxx{\Omega'}$ contains no integer points.
Our goal is to use this fact to show that for every facet $F$ of $P$, the vertex of the dual of $P$ corresponding to $F$ is a lattice point, i.e., to show that the primitive facet normal to $\cone{F}$ is given by $(\bp,1)$ for some lattice point $\bp$.
Let $F$ be a facet of $P$, and let $\Omega=\conv{(\br_1,q),\dots ,(\br_{m+1},q)}\in T$ be a full-dimensional simplex contained in $F$.
Since the origin lies in the interior of $P$, the dual of $P$ is a rational polytope containing the origin.
Further, the vector normal to $\cone{F}$ can be written in the form $(\bp,b)$ with $b>0$, where $\bp$ is an integer vector that is primitive, i.e., the greatest common divisor of the entries in $(\bp,b)$ equals $1$. 
Observe that $(\bp,b)\cdot (\mathbf{0},1)=b$.
If $b=1$, then the vertex of the dual of $P$ corresponding to $F$ is a lattice point, and our proof is complete.

Otherwise, suppose that $b>1$.
Since $(\bp,b)$ is primitive, there exists an integer vector $\bv$ such that $(\bp,b)\cdot \bv=1$.
Since $b>1>0$, $\bv$ is an element of the subset $S$ strictly contained between the hyperplane $H_0$ spanned by $\cone{F}$ and the affine hyperplane $H_b=H_0+(\mathbf{0},1)$; we can precisely describe this subset as
\[
  S:=\left\{\sum_{i=1}^{m+1}\lambda_i(\br_i,q)+\lambda(\mathbf{0},1):\lambda_i\in \R \text{ and }0<\lambda<1\right\} .
\]
Since $b(z)=0$, it follows that for each $\tau \subseteq \Omega$ the set $\boxx{\tau'}=\boxx{\tau,(\mathbf{0},1)}$ contains no integer points.
The key observation is that translates of $\bigcup_{\tau \subseteq \Omega}\boxx{\tau,(\mathbf{0},1)}$ by the integer ray generators of $\cone{F}$ cover $S$, though this union is not disjoint, i.e.,
\[
S=\bigcup_{\mu_1\,\ldots,\mu_{m+1}\in \Z}\left( \left(\sum_i\mu_i(\br_i,q)\right) + \bigcup_{\tau \subseteq \Omega}\boxx{\tau,(\mathbf{0},1)} \right) .
\]
This cover property follows from taking an arbitrary $\sum_{i=1}^{m+1}\lambda_i(\br_i,q)+\lambda(\mathbf{0},1)\in S$ and expressing each coefficient as a sum of an integer and fractional part.
It follows that $S$ contains no integer points, since $\bigcup_{\tau \subseteq \Omega}\boxx{\tau,(\mathbf{0},1)}$ contains no integer points.
Hence, no such integer vector $\bv$ exists, implying that $b=1$.
Since $F$ was arbitrary, it follows that the dual of $P$ is a lattice polytope. 
\end{proof}

\subsection{Reflexive Polytopes of Higher Index}
Kasprzyk and Nill~\cite{KasprzykNill} introduced the following class of polytopes .

\begin{definition}
A lattice polytope $P$ is a \emph{reflexive polytope of higher index} $\mathcal{L}$ (also
known as an \emph{$\mathcal{L}$-reflexive polytope}), for some $\mathcal{L}\in\Z_{>0}$, if the following conditions hold:
\begin{itemize}
    \item $P$ contains the origin in its interior;
    \item The vertices of $P$ are primitive, i.e., the line segment joining each vertex to $\mathbf{0}$ contains no other lattice points;
    \item For any facet $F$ of $P$ the local index $\mathcal{L}_F$ equals $\mathcal{L}$, i.e., the integral distance of $\mathbf{0}$ from the affine hyperplane spanned by $F$ equals $\mathcal{L}$. 
\end{itemize}
\end{definition}

The $1$-reflexive polytopes are the reflexive polytopes mentioned earlier in the section.
Kaspryzk and Nill proved that if $P$ is a lattice polytope with primitive vertices containing the origin in its interior then $P$ is $\mathcal{L}$-reflexive if and only if $\mathcal{L}P^*$ is a lattice polytope having only primitive vertices. In this case, $\mathcal{L} P^*$ is also $\mathcal{L}$-reflexive.

Kaspryzk and Nill investigated $\mathcal{L}$-reflexive polygons. 
In particular, they show that there is no $\mathcal{L}$-reflexive polygon of even index.  
Furthermore, they provide a family of $\mathcal{L}$-reflexive polygons arising for each odd index: 
$$P_\mathcal{L}=\conv{\pm(0,1),\pm (\mathcal{L},2), \pm (\mathcal{L},1)}.$$
We are interested in the dual of $P_\mathcal{L}$:
$$P^*_\mathcal{L}=\conv{\pm\left(\frac{1}{\mathcal{L}},0\right),\pm(\frac{2}{\mathcal{L}},-1),\pm\left(\frac{1}{\mathcal{L}},-1\right)}.$$
\vspace{0.25cm}
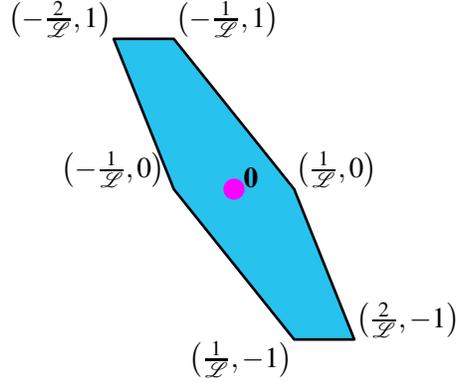
\begin{figure}[H]
\begin{center}

\begin{tikzpicture}[x  = {(2cm,0cm)},
                    y  = {(0cm,1cm)},
                    z  = {(0cm,0cm)},
                    scale = 2,
                    color = {lightgray}]

  \definecolor{pointcolor_p}{rgb}{1,0,1}
  \tikzstyle{pointstyle_p} = [fill=pointcolor_p]

  \coordinate (v0_p) at (0.2, 0);
  \coordinate (v1_p) at (-0.2, 0);
  \coordinate (v2_p) at (0.4, -1);
  \coordinate (v3_p) at (-0.4, 1);
  \coordinate (v4_p) at (0.2, -1);
  \coordinate (v5_p) at (-0.2, 1);
  \coordinate (0,0) at (0,0);

  \definecolor{edgecolor_p}{rgb}{ 0,0,0 }

  \definecolor{facetcolor_p}{rgb}{0.0, 0.72, 0.92}
  \tikzstyle{facestyle_p} = [fill=facetcolor_p, fill opacity=0.85, draw=edgecolor_p, line width=1 pt, line cap=round, line join=round]

  \draw[facestyle_p] (v1_p) -- (v4_p) -- (v2_p) -- (v0_p) -- (v5_p) -- (v3_p) -- (v1_p) -- cycle;

  \fill[pointcolor_p] (v1_p);
  \node at (v1_p) [text=black, inner sep=0.5pt, above left, draw=none, align=left] {$\left(-\frac{1}{\mathcal{L}},0\right)$     };
  \fill[pointcolor_p] (v4_p);
  \node at (v4_p) [text=black, inner sep=0.5pt, below left, draw=none, align=left] {$\left(\frac{1}{\mathcal{L}},-1\right)$};
  \fill[pointcolor_p] (v2_p);
  \node at (v2_p) [text=black, inner sep=0.5pt, above right, draw=none, align=left] {$\left(\frac{2}{\mathcal{L}},-1\right)$};
  \fill[pointcolor_p] (v0_p);
  \node at (v0_p) [text=black, inner sep=0.5pt, above right, draw=none, align=left] {$\left(\frac{1}{\mathcal{L}},0\right)$};
  \fill[pointcolor_p] (v5_p);
  \node at (v5_p) [text=black, inner sep=0.5pt, above right, draw=none, align=left] {$\left(-\frac{1}{\mathcal{L}},1\right)$};
  \fill[pointcolor_p] (v3_p);
  \node at (v3_p) [text=black, inner sep=0.5pt, above left, draw=none, align=left] {$\left(-\frac{2}{\mathcal{L}},1\right)$};
  draw=none, align=left] {5};
  \fill[pointcolor_p] (0,0) circle (2 pt);
  \node at (0,0) [text=black, inner sep=0.5pt, above right, draw=none, align=left] {\phantom{t}$\mathbf{0}$};


\end{tikzpicture}
\end{center}
\caption{The rational hexagon $P^*_\mathcal{L}$.}
\label{fig:hexagon} 
\end{figure}

Let $\mathcal{L}$ be odd. 
Our goal in the remainder of this subsection is to compute the $h^*$-polynomial of $P^*_\mathcal{L}$ using Theorem~\ref{h^*-polynomial}, to illustrate how this theorem can be applied.
Consider the boundary as its own triangulation $T$ (with denominator $\mathcal{L}$) of $P^*_\mathcal{L}$ and take the set of integral ray generators of $\cone{P^*_\mathcal{L}}$ to be $$\{\pm(1,0,\mathcal{L}),\pm(2,-\mathcal{L},\mathcal{L}),\pm(1,-\mathcal{L},\mathcal{L})\}.$$
Observe that $T$ contains six edges, six vertices, and the empty face $\emptyset$. 
It is not difficult to see that the box polynomials of the $0$-simplices are $0$. For example, in order for
\[ \boxx{(2,-\mathcal{L},\mathcal{L})}=\{\lambda_1(2,-\mathcal{L},\mathcal{L}):
0<\lambda_1<1\}\cap \Z^3 \]
to contain any lattice points, $2\lambda_1$ must be an integer between $0$ and $2$, implying that $\lambda_1=\frac{1}{2}$.
Also, $-\mathcal{L}\lambda_1$ and $\mathcal{L}\lambda_1$ must be integers, but since $\lambda_1=\frac{1}{2}$ and $\mathcal{L}$ is odd, $-\mathcal{L}\lambda_1$ and $\mathcal{L}\lambda_1$ are never integers. 
Therefore, $\boxx{(2,-\mathcal{L},\mathcal{L})}\cap \Z^3=\emptyset$.

Since $P^*_\mathcal{L}$ is a centrally symmetric hexagon, we can restrict our analysis to three of its facets: $F_1:=\conv{\pm\left(\frac{1}{\mathcal{L}},-1\right),\pm(\frac{2}{\mathcal{L}},-1)}$, $F_2:=\conv{\pm(\frac{2}{\mathcal{L}},-1), \pm\left(\frac{1}{\mathcal{L}},0\right)}$, and $F_3:=\conv{\pm\left(\frac{1}{\mathcal{L}},0\right), \pm\left(-\frac{1}{\mathcal{L}},1\right)}$.
We consider each facet separately.

\textbf{Case: $F_1$.}
Observe:
\begin{align*}
\boxx{(F_1,\mathcal{L})}&=\{\lambda_1(1,-\mathcal{L},\mathcal{L})+\lambda_2(2,-\mathcal{L},\mathcal{L}): 0<\lambda_1,\lambda_2<1\}\\
&=\{(\lambda_1+2\lambda_2, -\mathcal{L}\lambda_1-\mathcal{L}\lambda_2, \mathcal{L}\lambda_1+\mathcal{L}\lambda_2: 0<\lambda_1,\lambda_2<1\}.
\end{align*}

Let $\mathcal{L}=2k+1$ for $k\in \Z_{\geq 0}$. 
We now want to determine when $(A,-B,B)\in \boxx{(F_1,\mathcal{L})}$ is a lattice point. 
This reduces to solving a system of linear equations between $A$ and $B$. 
In order for $A$ to be an integer it must be $1$ or $2$. 
When $A=\lambda_1+2\lambda_2=1$, $B=\mathcal{L}\lambda_1+\mathcal{L}\lambda_2$ equals $\mathcal{L}-k$, $\mathcal{L}-k+1$,\dots, $\mathcal{L}-2$, or $\mathcal{L}-1$ with the restriction that $0<\lambda_1,\lambda_2<1$.  
When $A=\lambda_1+2\lambda_2=2$, $B=\mathcal{L}\lambda_1+\mathcal{L}\lambda_2$ equals $\mathcal{L}+1$, $\mathcal{L}+2$,\dots, $\mathcal{L}+k-1$, or $\mathcal{L}+k$. 
Therefore, $\boxx{(F_1,\mathcal{L})}\cap \Z^3$ contains the elements $\{(1,k-\mathcal{L},\mathcal{L}-k),(1,k-\mathcal{L}-1,\mathcal{L}-k+1),\dots(1,2-\mathcal{L},\mathcal{L}-2),(1,1-\mathcal{L},\mathcal{L}-1),(2,-\mathcal{L}-1,\mathcal{L}+1),(2,-\mathcal{L}-2,\mathcal{L}+2),\dots,(2,1-\mathcal{L}-k,\mathcal{L}+k+1),(2,-\mathcal{L}-k,\mathcal{L}+k)\}$.
Therefore, the box polynomial of $F_1$ is
$$
B(F_1;z)=\sum_{i=\mathcal{L}-k}^{\mathcal{L}-1}z^i+\sum_{i=\mathcal{L}+1}^{\mathcal{L}+k}z^i \, .
$$
\newpage
\textbf{Case: $F_2$.}
Observe:

\begin{align*}
\boxx{F_2,\mathcal{L}}&=\{\lambda_1(2,-\mathcal{L},\mathcal{L})+\lambda_2(1,0,\mathcal{L}):0<\lambda_1,\lambda_2<1\}\\
&=\{(2\lambda_1+\lambda_2, -\mathcal{L}\lambda_1,\mathcal{L}\lambda_1+\mathcal{L}\lambda_2):0<\lambda_1,\lambda_2<1\}.
\end{align*}
Suppose $(A,B,C)$ is an integer point in this set.
Again, determining the integer points in the box reduces to solving a system of linear equations between $A$ and $C$ with the added condition coming from $B$ that $\lambda_1=\frac{1}{\mathcal{L}},\dots, \frac{\mathcal{L}-1}{\mathcal{L}}$.  
It is straightforward to verify that the resulting box polynomial of $F_2$ is the same as $F_1$.

\textbf{Case: $F_3$.}
Observe:
\begin{align*}
\boxx{F_3,\mathcal{L}}&=\{\lambda_1(-1,\mathcal{L},\mathcal{L})+\lambda_2(1,0,\mathcal{L}):0<\lambda_1,\lambda_2<1\}\\
&=\{(-\lambda_1+\lambda_2,\mathcal{L}\lambda_1, \mathcal{L}\lambda_1+\mathcal{L}\lambda_2):0<\lambda_1,\lambda_2<1\}.
\end{align*}

Suppose $(A,B,C)$ is an integer point in this set.
For $A$ to be an integer it must be equal to zero, so we obtain $\lambda_1=\lambda_2$. 
The expression for $B$ implies that $\lambda_1=\frac{m}{\mathcal{L}}$ for some integer $m\in [1,\mathcal{L}-1]$.
Lastly, $C$ then reduces to $2\mathcal{L}\lambda_1=2m$.
Therefore, we conclude $\boxx{(F_3,\mathcal{L})}$ contains $\mathcal{L}-1$ lattice points of the form $(0,m,2m)$, one for each integer $m\in [1,\mathcal{L}-1]$. 
This implies the box polynomial of $F_3$ is given by
$$
B(F_3;z)=\sum_{i=1}^{\mathcal{L}-1}z^{2i}\, .
$$

 \begin{table}
    \centering
 \begin{tabular}{||c c c c||} 
 \hline
 $\Omega\in T$ & $\dim(\Omega)$ & $B(\Omega;z)$ & $h(\Omega,z^\mathcal{L})$ \\ [0.5ex] 
 \hline\hline
 $F_1$ & 1 & $\sum_{i=\mathcal{L}-k}^{\mathcal{L}-1}z^i+\sum_{i=\mathcal{L}+1}^{\mathcal{L}+k}z^i$ & 1 \\ 
 \hline
 $-F_1$ & 1 & $\sum_{i=\mathcal{L}-k}^{\mathcal{L}-1}z^i+\sum_{i=\mathcal{L}+1}^{\mathcal{L}+k}z^i$ & 1 \\
 \hline
 $F_2$ & 1 & $\sum_{i=\mathcal{L}-k}^{\mathcal{L}-1}z^i+\sum_{i=\mathcal{L}+1}^{\mathcal{L}+k}z^i$ & 1 \\
 \hline
 $-F_2$ & 1 & $\sum_{i=\mathcal{L}-k}^{\mathcal{L}-1}z^i+\sum_{i=\mathcal{L}+1}^{\mathcal{L}+k}z^i$ & 1 \\
 \hline
 $F_3$ & 1 & $\sum_{i=1}^{\mathcal{L}-1}z^{2i}$ & 1 \\ 
 \hline
 $-F_3$ & 1 & $\sum_{i=1}^{\mathcal{L}-1}z^{2i}$ & 1 \\ 
 \hline
 $\left(\frac{1}{\mathcal{L}},0\right)$ & $0$ & $0$ & $1+z^\mathcal{L}$ \\ 
   \hline
  $\left(-\frac{1}{\mathcal{L}},0\right)$& $0$ & $0$ & $1+z^\mathcal{L}$ \\ 
   \hline
 $\left(\frac{2}{\mathcal{L}},-1\right)$ & $0$ & $0$ & $1+z^\mathcal{L}$ \\ 
   \hline
 $\left(-\frac{2}{\mathcal{L}},1\right)$ & $0$ & $0$ & $1+z^\mathcal{L}$ \\ 
   \hline
 $\left(\frac{1}{\mathcal{L}},-1\right)$ & $0$ & $0$ & $1+z^\mathcal{L}$ \\ 
   \hline
 $\left(-\frac{1}{\mathcal{L}},1\right)$ & $0$ & $0$ & $1+z^\mathcal{L}$ \\ 
   \hline
 $\emptyset$ & -1 & 1 & $1+4z^{\mathcal{L}}+z^{2\mathcal{L}}$ \\ [1ex] 
 \hline
 \end{tabular}
 \label{tab:table}
 \caption{}
\end{table}

Combining the above analysis with the values in Table~\ref{tab:table}, we apply Theorems~\ref{h^*-polynomial} and~\ref{prop: dual} and conclude that for $\mathcal{L}=2k+1$,
$$h^*(P^*_\mathcal{L};z)=(1+z+\cdots+z^{\mathcal{L}})\left(1+4z^\mathcal{L}+z^{2\mathcal{L}}+4\left(\sum_{i=\mathcal{L}-k}^{\mathcal{L}-1}z^i+\sum_{i=\mathcal{L}+1}^{\mathcal{L}+k}z^i\right)+2\sum_{i=1}^{\mathcal{L}-1}z^{2i}\right).$$

\vspace{5mm}
\section*{Acknowledgements}
This work was partially supported by NSF Graduate Research Fellowship DGE-1247392 (ARVM).  
ARVM thanks the Discrete Geometry group of the Mathematics Institute at FU Berlin for providing a wonderful working environment while part of this work was done.
The authors would like to thank Steven Klee, Jos\'e Samper, Liam Solus, and three anonymous
referees for fruitful correspondence. 


\bibliographystyle{amsplain}
\bibliography{VindasMelendez}

\end{document}